\documentclass[12pt]{amsart}
\usepackage{amssymb, amscd, amsmath, amsthm, latexsym, enumerate}

\newtheorem{theorem}{Theorem}
\newtheorem{lemma}[theorem]{Lemma}

\newtheorem*{thm}{Theorem}

\begin{document}

\title{sections of surface bundles}

\author{Jonathan A. Hillman }
\address{School of Mathematics and Statistics\\
     University of Sydney, NSW 2006\\
      Australia }

\email{jonathan.hillman@sydney.edu.au}

\begin{abstract}
Let $p:E\to{B}$ be a bundle projection with base $B$ and fibre $F$ 
aspherical closed connected surfaces.
We review what algebraic topology can tell us 
about such bundles and their total spaces\
and then consider criteria for $p$ to have a section.
In particular, we simplify the cohomological obstruction,
and show that the transgression $d^2_{2,0}$ in the homology LHS 
spectral sequence of a central extension is evaluation of the extension class.
We also give several examples of bundles without sections.
\end{abstract}



\maketitle

Let $p:E\to{B}$ be a bundle projection with base $B$ and fibre $F$ 
aspherical closed connected surfaces,
and let $\pi=\pi_1(E)$, $\beta=\pi_1(B)$ and $\phi=\pi_1(F)$.
The exact sequence of homotopy for $p$ reduces to an extension 
of fundamental groups
\[
\xi(p):\quad1\to\phi\to\pi\to\beta\to1,
\]
which determines $p$ up to bundle isomorphism over $B$.
Sections \S2--\S4 are a summary of our talk at the Bonn conference, 
reviewing what else algebraic topology can tell us about such bundles 
and their total spaces.
(We refer to \cite{Hi} for the arguments and other sources.)

The main part of this work (\S5--\S8)
considers criteria for the existence of sections.
Such a bundle $p$ has a section if and only if $\xi(p)$ splits,
and this is so if and only if the action of $\beta$ through 
outer automorphisms of $\phi$ lifts, and a cohomology class is trivial.
We simplify the latter condition,
and identify the transgression $d^2_{2,0}$ in the homology LHS spectral
sequence of a central extension with evaluation of the extension class.
It is relatively easy to find examples of torus bundles without sections,
but seems more difficult to construct such examples with hyperbolic fibre.
We thank H.Endo for the example with base and fibre of genus 3 given in \S9,
and N.Salter and the referees for further suggestions relating to this example.

We conclude with a short list of questions arising from issues considered here.

\section{notation}

Let $\zeta{G}$, $G'$ and $I(G)$ denote the centre, 
the commutator subgroup and the isolator subgroup of a group $G$, 
respectively.
(Thus $G'\leq{I(G)}$ and $G/I(G)$ is the maximal torsion-free 
quotient of the abelianization $G^{ab}=G/G'$.) 
If $H$ is a subgroup of $G$ let $C_G(H)$ be the centralizer of $H$ in $G$.
Let $c_g$ denote conjugation by $g$, for all $g\in{G}$.
If $S$ is a subset of $G$ then $\langle{S}\rangle$ is the subgroup 
of $G$ generated by $S$ and $\langle\langle{S}\rangle\rangle$ 
is the normal closure of $S$ in $G$
(the smallest normal subgroup of $G$ which contains $S$).

If $A$ is an abelian group and $\theta:\beta\to{Aut}(A)$
is a homomorphism then $A^\theta$ shall denote the left $\mathbb{Z}[\beta]$-module with underlying group $A$ and module structure determined by $g.a=\theta(g)(a)$ for all $g\in\beta$ and $a\in{A}$.
(We shall usually write just $A$ when the action $\theta$ is trivial.)
If $G$ is any group then the homomorphism from $Aut(G)$ to $Aut(\zeta{G})$ determined by restriction factors through $Out(G)$.
In particular,
a homomorphism $\theta:\beta\to{Out}(G)$ determines a 
$\mathbb{Z}[\beta]$-module $\zeta{G}^\theta$.

We shall assume throughout that ``surface" means aspherical 
closed connected 2-manifold, 
except in the result cited at the end of \S2.
(We do not assume that surfaces are orientable,
although this constraint is imposed by some of the references cited.)
A group $G$ is a $PD_2$-group if $G\cong\pi_1(X)$ for some
such surface $X$, and it is a $PD_2^+$-group if $X$ is orientable.
If $G$ is a $PD_2$-group, with orientation character $w=w_1(G)=w_1(X)$,
let $G^+=\mathrm{Ker}(w)$ and let $X^+$ be the associated 
orientable covering space of $X$.

\section{bundles and group extensions}

The classification of surface bundles follows from the deep result of \cite{EE},
that if $X$ is a hyperbolic surface then the identity component 
of $Diff(X)$ is contractible.
(See \cite{Gr} for a proof using only differential topology, 
which applies also to the based case.)
The flat surfaces $X=T$ or $Kb$ have circle actions, 
and $Diff(X)_o\sim(S^1)^r$, 
where $r=2$ or 1 is the rank of the centre of $\pi_1(X)$.
In all cases, the inclusions of $Diff(X)$ into $Homeo(X)$ and into 
the group of self homotopy equivalences $E(X)$ are homotopy equivalences.

The numbers in parentheses before the statements of theorems 
in this section refer to the corresponding theorems in \cite{Hi}.

\begin{thm} [5.2]
Let $p:E\to{B}$ be a bundle projection with base $B$ and fibre $F$ 
surfaces.
Then $p$ is determined up to bundle isomorphism over $B$
by the group extension $\xi(p)$.
Conversely, every such extension is realized by some bundle.
\qed
\end{thm} 

Conjugation in $\pi$ determines a homomorphism $\theta$ 
from $\beta$ to the outer automorphism group $Out(\phi)$.
If $\chi(F)<0$ (i.e., if  $\zeta\phi=1$), 
the extension is determined by the action alone; 
in general, 
extensions corresponding to a given action $\theta$ 
are classified by characteristic cohomology classes in
$H^2(B;\zeta\phi^\theta)=H^2(\beta;\zeta\phi^\theta).$

When is a closed 4-manifold $M$ ``equivalent" 
to the total space of such a bundle,
and, if so, in how many ways?
If equivalent means ``homotopy equivalent" or ``TOP $s$-cobordant", 
there is a satisfactory answer,
but little is known about diffeomorphism, 
except when $M$ has additional structure.

\begin{thm} [3.5.1]
Let $M$ be a closed $4$-manifold such that $\pi=\pi_1(M)$ 
is an extension of $\pi_1(B)$ by $\pi_1(F)$, 
where $B$ and $F$ are surfaces.
Then $M$ is aspherical if and only if $\chi(M)=\chi(B)\chi(F)$.
\qed
\end{thm}

Although 4-dimensional TOP surgery techniques are not yet available if 
$\pi$ has non-abelian free subgroups, 
5-dimensional surgery often suffices to construct $s$-cobordisms. 

\begin{thm} [6.15]
A closed $4$-manifold $M$ is TOP $s$-cobordant to the total space $E$ 
of an $F$-bundle over $B$,
where $B$ and $F$ are surfaces,
if and only if $\pi_1(M)\cong\pi_1(E)$ and $\chi(M)=\chi(E)$.
If so, then the universal covering space $\widetilde M$ 
is homeomorphic to $\mathbb{R}^4$.
\qed
\end{thm} 
                      
When $\pi$ is solvable $s$-cobordism implies homeomorphism,
and $M$ is then homeomorphic to an
$\mathbb{E}^4$-, $\mathbb{N}il^4$-, $\mathbb{N}il^3\times\mathbb{E}^1$- 
or $\mathbb{S}ol^3\times\mathbb{E}^1$-manifold.
Conversely, if $M$ has one of these geometries and $\beta_1(M)\geq2$
then $M$ fibres over $T$. (See Chapter 8 of \cite{Hi}.)
The other geometries realized by total spaces of surface bundles are
$\mathbb{H}^2\times\mathbb{E}^2$, 
$\mathbb{H}^3\times\mathbb{E}^1$,
$\widetilde{\mathbb{SL}}\times\mathbb{E}^1$ and
$\mathbb{H}^2\times\mathbb{H}^2$.
Hamenst\"adt has announced that no such bundle space 
has geometry $\mathbb{H}^4$ \cite{Ha}.
Finally, $\mathbb{H}^2(\mathbb{C})$ may be excluded as a consequence 
of the next theorem (due independently to Kapovich, Kotschick and Hillman) 
and the fact that quotients of the unit ball in $\mathbb{C}^2$
do not submerse holomorphically onto complex curves \cite{Liu}.

\begin{thm} [13.7]
Let $S$ be a complex surface.
Then $S$ admits a holomorphic submersion onto a complex curve, 
with base and fibre of genus $\geq2$, if and only if 
$S$ is homotopy equivalent to the total space of a bundle
with base and fibre hyperbolic surfaces.
\qed
\end{thm}

In this case homotopy equivalence implies diffeomorphism!

\section{the group determines the bundle up to finite ambiguity}

If $\chi(B)<0$, then there are only finitely many ways of representing
$\pi$ as an extension of $PD_2$-groups, 
up to ``change of coordinates" \cite{Jo}.
Let $\pi$ be a $PD_4$-group with a normal subgroup $K$ 
such that $K$ and $\pi/K$ are surface groups with trivial centre.
Johnson showed that whether 
\renewcommand{\theenumi}{\Roman{enumi}}
\begin{enumerate}

\item $\mathrm{Im}(\theta)$ is infinite and $\mathrm{Ker}(\theta)\not=1$;

\item $\mathrm{Im}(\theta)$ is finite; or

\item $\theta$ is injective
\end{enumerate} 
\renewcommand{\theenumi}{\arabic{enumi}}
depends only on $\pi$ and not on the subgroup $K$.
In case I there is an unique such normal subgroup, 
and in case II there are two, and $\pi$ is virtually a product.
In case III there are finitely many such normal subgroups.
These assertions follows ultimately from the facts 
that nontrivial finitely generated normal subgroups
of hyperbolic surface groups have finite index, 
and the Euler characteristic increases on passage to such subgroups.
(See \cite{Jo}, or Theorems 5.5 and 5.6 of \cite{Hi}.)

There are only finitely many conjugacy classes of finite subgroups of $Out(\phi)$, for any given $PD_2$-group $\phi$.
(This follows from the Nielsen Conjecture.
See Chapter 7 of \cite{FM} for the case when $\phi$ is orientable.)
Thus the number of groups $\pi$ of type II with given $\chi(\pi)$ is finite.
In particular, 
for each $n>0$ there are only finitely many surface bundles 
with total space $E$ such that $\chi(E)=n$ and $E$ admits a geometry 
(necessarily $\mathbb{H}^2\times\mathbb{H}^2$ -- see Chapter 13 of \cite{Hi}.)
At the other extreme, if $\phi$ is orientable and
$\pi$ has no subgroup isomorphic to $\mathbb{Z}^2$, 
it must be of type III, 
and there are again only finitely many such groups with given $\chi(\pi)$
\cite{Bo}.

The examples of Kodaira, Atiyah and Hirzebruch of surface bundles 
with nonzero signature are of type III, 
and each have at least two such normal subgroups \cite{BD}.
It is noteworthy that in each case one of the subgroups 
satisfies the condition $\chi(K)^2\leq\chi(\pi)$.
It is a straightforward consequence of Johnson's arguments 
that this extra condition holds for at most one such subgroup $K$, 
if $\pi$ is of type III.
(In particular, if $\chi(\pi)=4$ there is at most one such $K$ 
with $K$ and $\pi/K$ both orientable.)

The Johnson trichotomy extends to the case when $\pi/K$ has a centre,
but is inappropriate if $\zeta K\not=1$,
as there are then nontrivial extensions with trivial action ($\theta=1$).
Moreover $Out(K)$ is then virtually free, and so $\theta$ is never injective.

However the situation is very different if $\pi/K\cong\mathbb{Z}^2$, 
$\chi(K)<0$ and $\beta_1(\pi)>2$.
For then there are epimorphisms from $\pi$ to $\mathbb{Z}^2$ 
with kernel a surface group of arbitrarily high genus \cite{Bu}.

\section{orbifolds, seifert fibrations and virtual bundles}

In this section we allow the base $B$ to be an aspherical 2-orbifold.
An {\it orbifold bundle with general fibre $F$ over $B$} is a map 
$f:M\to B$ which is locally equivalent to a projection
 ${G\backslash(F\times D^2)}\to G\backslash {D^2}$, 
where $G$ acts freely on $F$ and effectively and orthogonally on $D^2$.
We shall also say that $f:M\to{B}$ is an $F$-{\it orbifold bundle\/} (over $B$)
and $M$ is an $F$-{\it orbifold bundle space}.

A {\it Seifert fibration\/} (in dimension 4) 
is an orbifold bundle with general fibre $T$ or $Kb$.
A {\it virtual bundle space\/} is a manifold with a finite cover 
which is the total space of a bundle. 
Orbifold bundle spaces are virtual bundle spaces,
but the converse is not true.
(See page 189 of \cite{Hi}.)

Aspherical orbifold bundles (with 2-dimensional base and fibre)
are determined up to fibre-preserving diffeomorphism 
by their fundamental group sequences.
In many cases they are determined up to diffeomorphism
(among such spaces) by the group alone \cite{Vo}. 
See \cite{Hi11} for a discussion of geometries and geometric decompositions
of the total spaces of such orbifold bundles.

Johnson's trichotomy extends to groups commensurate 
with extensions of a surface group by
a surface group with trivial centre,
but it is not known whether all  torsion-free such groups are realized by aspherical 4-manifolds.

\section{extensions of groups}

The extensions $\xi$ of a group $\beta$ with kernel $\phi$
and action $\theta:\beta\to{Out}(\phi)$
(induced by conjugation in the ``ambient group" $\pi$)
may be parametrized by $H^2(\beta;\zeta\phi^\theta)$.
In general, there is an obstruction in $H^3(\beta;\zeta\phi^\theta)$ 
for there to be such an extension, 
but this obstruction group is trivial when $\beta$ is a surface group.
If $\theta$ factors through a homomorphism 
$\tilde\theta:\beta\to{Aut(\phi)}$ 
then the semidirect product $\phi\rtimes_{\widetilde\theta}\beta$ 
corresponds to $0\in{H^2}(\beta;\zeta\phi^\theta)$.
(See Chapter IV of \cite{Br}.) 

Let $p:\pi\to\beta$ be an epimorphism with kernel $\phi$ 
corresponding to such an extension $\xi$.
Then $\xi$ {\it splits\/} if there is a homomorphism $s:\beta\to\pi$
such that $ps=id_\beta$.
(If so, then $\pi$ is a semidirect product.)
This is so if and only if $\theta$ factors through $Aut(\phi)$ 
and the cohomology class $[\xi]\in{H^2}(\beta;\zeta\phi^\theta)$ 
of the extension is 0.

\begin{lemma}
If $\zeta\phi=1$ then $\xi$ splits if and only if the action $\theta$ factors through $Aut(\phi)$.
\end{lemma}

\begin{proof}
If $\phi$ has trivial centre then the extension is determined by the action,
since $H^2(\beta;\zeta\phi)=0$.
Thus if the action factors $\pi$ must be a semidirect product,
i.e., $p_*$ splits.
The converse is clear.
\end{proof}

The exact sequence of low degree for the extension has the form
\[
H_2(\pi;\mathbb{Z})\to{H_2(\beta;\mathbb{Z})}\to
{H_0}(\beta;H_1(\phi;\mathbb{Z})^{\theta^{ab}})\to{H_1(\pi;\mathbb{Z})}\to
{H_1(\beta;\mathbb{Z})}\to0,
\]
where $\theta^{ab}$ is the automorphism of $\phi^{ab}$ 
induced by $\theta$.
This exact sequence is usually derived from the
homology LHS spectral sequence for the extension.
The second homomorphism in this sequence is the {\it transgression} 
$d^2_{2,0}$, 
the first non-trivial differential on page 2 of the spectral sequence.
In \S7 we shall show that when $\phi$ is abelian then $d^2_{2,0}$
is the image $\xi_*$ of $[\xi]$ under the 
change of coefficients and evaluation homomorphisms
\[
H^2(\beta;\zeta\phi^\theta)\to{H^2(\beta;H_0(\beta;\zeta\phi^\theta))}\to
{Hom}(H_2(\beta;\mathbb{Z}),H_0(\beta;\zeta\phi^\theta)).
\]
(The transgression in degree $q$ for a fibration $p:E\to{B}$ with fibre $F$ 
was originally defined as a homomorphism from a subgroup 
of $H_q(B;\mathbb{Z})$ to a quotient of $H_{q-1}(F;\mathbb{Z})$,
corresponding to the connecting homomorphism from $\pi_q(B)$ to $\pi_{q-1}(F)$
in the long exact sequence of homotopy, via the Hurewicz homomorphisms.
It may be identified with a differential on page $q$ 
of the homology spectral sequence for the fibration,
and the terminology has been dualized to cohomology 
and extended to purely algebraic situations.
See page 172 of \cite{Mc}.)

If there is a homomorphism $s:\beta\to\pi$ which splits $p$ 
then $H_i(s)$ splits $H_i(p)$, for all $i$, 
and so the five-term exact sequence above gives rise to an isomorphism
\[
\pi^{ab}\cong(\phi^{ab}/(I-\theta^{ab})\phi^{ab})\oplus\beta^{ab}=
(\phi/[\pi,\phi])\oplus\beta^{ab}.
\]
This apparently innocuous observation gives the most practical test for whether
an extension $\xi$ splits, both when $\phi'=1$, as in the next lemma,
and when $\zeta\phi=1$, as considered in \S9 below.

\begin{lemma}
Let $G$ be a group with a finitely generated abelian normal
subgroup $A$ such that $\beta=G/A$ is a $PD_2^+$-group.
Then the canonical projection from $G$ to $\beta$ 
has a section if and only if 
\[
G^{ab}\cong{A/[G,A]}\oplus\beta^{ab}.
\]
\end{lemma}

\begin{proof}
Let $\theta:\beta\to{Aut}(A)$ be the action induced by conjugation in $G$.
Let $\overline{A}=A/[G,A]$ and $\overline{G}=G/[G,A]$.
Then $\overline{G}$ is a central extension of $\beta$ by $\overline{A}$,
and $G^{ab}=\overline{G}^{ab}$.
Since $c.d.\beta=2$, the epimorphism from $A$ to $\overline{A}$
induces an epimorphism from $H^2(\beta;A^\theta)$ to $H^2(\beta;\overline{A})$.
Since $\beta$ is a $PD_2^+$-group, 
$H^2(\beta;A^\theta)\cong{H_0(\beta;A^\theta)}$
and $H^2(\beta;\overline{A})\cong{H_0(\beta;\overline{A})}$.
These are each isomorphic to $\overline{A}$,
and so the natural homomorphism  from
$H^2(\beta;A^\theta)$ to $H^2(\beta;\overline{A})$ is an isomorphism.
Therefore $G$ splits as a semidirect product if and only if 
the same is true for $\overline{G}$.
Since $\overline{A}$ is central in $\overline{G}$, 
this is so if and only if $\overline{G}\cong\overline{A}\times\beta$,
and this is equivalent to
$\overline{G}^{ab}\cong\overline{A}\oplus\beta^{ab}$. 
\end{proof}

\section{the extension class for actions which lift}

Let $p:\pi\to\beta$ be an epimorphism with kernel $\phi$,
and suppose that
$\beta$ has a finite presentation $\langle{X}|R\rangle$, 
with associated epimorphism $q:F(X)\to\beta$.
Let $\varepsilon:\mathbb{Z}[\beta]\to\mathbb{Z}$ be the augmentation homomorphism,
and let $\partial_x:\mathbb{Z}[F(X)]\to\mathbb{Z}[\beta]$ be the composite
of the Fox free derivative with the linear extension of $q$, for each $x\in{X}$.
Then $\varepsilon_x(v)=\varepsilon\partial_x(v)$ is the exponent sum of $x$ 
in the word $v$.
The presentation determines a Fox-Lyndon partial resolution 
\[
C_*^{FL}(\beta):\quad \mathbb{Z}[\beta]^R\to\mathbb{Z}[\beta]^X\to\mathbb{Z}[\beta]\to\mathbb{Z}\to0,
\] 
where the differentials map the basis elements by
\[
\partial{c_0}=1,\quad\partial{c_1^x}=q(x)-1\quad\mathrm{ and}\quad \partial{c_2^r}=\Sigma_{y\in{X}}\partial_yr.c_1^y,
\] for all $r\in{R}$ and $x\in{X}$.

Let $h_*:C_*^{FL}(\beta)\to{C_*^{bar}}$ be the chain morphism to
the standard (bar) resolution $C_*^{bar}$ given by 
the identity on $C_0^{FL}(\beta)=\mathbb{Z}[\beta]=C_0^{bar}$,
and which sends the basis elements $c_21^x$ of 
$C_1^{FL}(\beta)=\mathbb{Z}[\beta]^X$ and
$c_2^r$ of $C_2^{FL}(\beta)=\mathbb{Z}[\beta]^R$ 
to $[q(x)]$ and to $\Sigma_{x\in{X}}[\partial_xr|q(x)]\in{C_2^{bar}}$, 
respectively, for all $x\in{X}$ and $r\in{R}$.
(Here the symbol $[-|-]$ is extended to be additive in its first argument.
See Exercises II.5.3 and  II.5.4 of \cite{Br}.)
If $c.d.\beta\leq2$ then $C_*^{FL}(\beta)$ is a resolution and $h$ is a chain homotopy equivalence.

Suppose that $\theta$ factors through a homomorphism  $\widetilde\theta:\beta\to{Aut}(\phi)$.
Let $\sigma:\beta\to\pi$ be a set-theoretic section such that $\sigma(1)=1$
and $c_{\sigma(g)}=\widetilde\theta(g)$, for all $g\in\beta$,
and define a function $f:\beta^2\to\pi$ by
\[
\sigma(g)\sigma(h)=f([g|h])\sigma(gh),\quad\mathrm{for~all}~g,h\in\beta.
\]
Then $f([g|h])$ is in $\zeta\phi$, for all $g,h\in\beta$,
since $\widetilde\theta$ is a homomorphism.
The linear extension $f:C_2^{bar}\to\zeta\phi$ is a 2-cocycle,
which represents  the extension class $[\xi]\in{H^2}(\beta;\zeta\phi)$.
(See \S3 and \S6 of Chapter IV of \cite{Br}.)

For each $r=\Pi_{i=1}^cx_i^{\eta(i)}$ in $R$, with $\eta(i)=\pm1$,
let $I_k(r)=\Pi_{i=1}^{k-1}x_i^{\eta(i)}$, for $1\leq{k}\leq{c}$.
(If $k=1$ this is the empty product: $I_1(r)=1$.)
Then $\partial_xr=\Sigma_{x_i=x}\eta(i)q(I_i(r)x_i^{\delta(i)})$, 
where $\delta(i)=0$ if $\eta(i)=1$ and $\delta(i)=-1$ if $\eta(i)=-1$,
for all $x\in{X}$.
Hence\[
f(h_2(c_2^r))=f(\Sigma_{x\in{X}}[\partial_xr|q(x)])=
f(\Sigma_{i=1}^c[\eta(i)q(I_i(r)x_i^{\delta(i)})|q(x_i)])
\]
\[=
\Sigma_{i=1}^c\eta(i)f([q(I_i(r)x_i^{\delta(i)})|q(x_i)]).
\]
As this lies in $\pi$, we may write it multiplicatively as
\[
\Pi_{i=1}^cf([q(I_i(r)x_i^{\delta(i)})|q(x_i)])^{\eta(i)}.
\]
On the other hand, if ${s:F(X)\to\pi}$ is the homomorphism defined by $s(x)=\sigma(q(x))$, 
for all $x\in{X}$, then
\[s(r)=\Pi_{i=1}^cs(x_i)^\eta(i)=\Pi_{i=1}^c\sigma(q(x_i))^{\eta(i)}.
\]
A finite induction shows that this product equals 
\[
(\Pi_{i=1}^cf(q([I_i(r)x_i^{\delta(i)})|q(x_i)])^{\eta(i)})\sigma(q(r))
=\Pi_{i=1}^cf([q(I_i(r)x_i^{\delta(i)})|q(x_i)])^{\eta(i)}).
\]
(Note that $\sigma(g)\sigma(g^{-1})=f([g|g^{-1}])$,
and so $\sigma(g)^{-1}=f([g|g^{-1}])^{-1}\sigma(g^{-1})$,
for all $g\in\beta$.
The calculations simplify if the exponents $\eta(i)$ are all positive.)
Hence $s(r)=f(h_2(c_2^r))$, and so is in $\zeta\phi$, for all $r\in{R}$.
It follows that 
\[
\xi_*([z])=s(\Pi{r^{n_r}}),
\]
for any 2-cycle $z=\Sigma_{r\in{R}}{n_r}c^r_2$ of 
$\mathbb{Z}\otimes_\beta{C_*^{FL}(\beta)}$.

Suppose now that $\beta$ is a $PD_2$-group with 
1-relator presentation $\langle{X}|r\rangle$ 
and orientation character $w=w_1(\beta):\beta\to\mathbb{Z}^\times$.
Let $\varepsilon_w:\mathbb{Z}[\beta]\to\mathbb{Z}$ be the 
linear extension of $w$,
and let $J_w=\mathrm{Ker}(\varepsilon_w)$.
If $A$ is any left $\mathbb{Z}[\beta]$-module then
$H^2(\beta;A)\cong{A}/(\partial_xr|x\in{X})A$, since $c.d.\beta=2$.
This is isomorphic to $H_0(\beta;\overline{A})={A}/J_wA$, 
by Poincar\'e duality,
and so $J_w$ is also the ideal generated by $\{\partial_xr|x\in{X}\}$.
Then we may recapitulate the above discussion as follows.

\begin{lemma}
If $\theta$ factors through $Aut(\phi)$ then $s(r)$ is in $\zeta\phi$, 
and its image $[s(r)]\in\zeta\phi/J_w\zeta\phi$ is well defined.
The epimorphism $p_*$ splits if and only if $[s(r)]=0$.
\end{lemma}

\begin{proof}
The first assertion holds since $q(r)=1$.
If $\sigma'$ is another such set-theoretic section and $s'=\sigma'q$ 
then $s'(x)=u(q(x))s(x)$ for some function $u:q(X)\to\zeta\phi$. 
(Conversely, every such function $u$ arises in this way.)
A simple induction shows that $s'(r)=s(r)+\Sigma_{x\in{X}}\partial_xr.u(q(x))$,
and so $[s(r)]$ is independent of the choice of section.

If $\sigma:\beta\to\pi$ splits $p_*$ then we may take $s=\sigma{q}$,
and so $s(r)=1$ in $\phi$. Hence $[s(r)]=0$.
Conversely, if $[s(r)]=0$ then we may choose $s$ so that $s(r)=1$,
and so $p_*$ splits.
\end{proof}

\section{abelian extensions and transgression}

We shall show that if $\xi$ is an extension with abelian kernel $\phi$
then the transgression homomorphism $d^2_{2,0}$
in the associated exact sequence of low degree is the image 
of the class $[\xi]$ in $Hom(H_2(\beta;\mathbb{Z}),H_0(\beta;\zeta\phi))$.
The significance of this result is that it is often useful to have 
``explicit" expressions for homomorphisms defined via the 
machinery of homological algebra.
The result appears to be ``folklore", 
but we have not found a published proof.
(Theorem 4 of \cite{HS} gives the cohomological analogue.)
Our argument uses naturality of the constructions arising (cf. \cite{Ba,Be}) 
to reduce to a special case.

\begin{theorem}
Let $\xi$ be an extension of a group $\beta$ with abelian kernel $\phi$.
Then $d^2_{2,0}=\xi_*~\mathrm{:}$ the first non-trivial differential 
on page 2 of the LHS spectral sequence for the extension 
is the homomorphism given by evaluation of the extension class.
\end{theorem}

\begin{proof}
We shall reduce to the situation when $\beta\cong\mathbb{Z}^2$,
$\phi$ is infinite cyclic and central, and $[\xi]$ generates 
$H^2(\beta;\mathbb{Z})\cong\mathbb{Z}$.
(In this case $\pi\cong{F(2)/F(2)_{[3]}}$ is the free 2-generator 
nilpotent group of class 2.)

Let $\bar\xi$ be the extension $0\to\phi/[\pi,\phi]\to\pi/[\pi,\phi]\to\beta\to1$
obtained by factoring out  $[\pi,\phi]$.
Then the projection of $\pi$ onto $\pi/[\pi,\phi]$ induces isomorphisms of the
5-term exact sequences corresponding to the extensions $\xi$ and $\bar\xi$, 
and so we may assume that $\phi$ is central.

Secondly, 
every class in $H_2(\beta;\mathbb{Z})$ is the image 
of the fundamental class of an orientable surface.
(This is most easily seen topologically, 
by assembling pairwise the 2-simplices of a representative 2-cycle
for $H_2(K(\beta,1);\mathbb{Z})$, or by using orientable bordism,
since $\Omega_2(X)=H_2(X;\Omega_0)$ for any cell complex $X$.
However, there is an algebraic argument in \cite{Zi}.)
Thus if  $[z]\in{H_2}(\beta;\mathbb{Z})$ there is a $PD_2^+$-group $\hat\beta$ with fundamental class $[\hat\beta]\in{H_2}(\hat\beta;\mathbb{Z})$ and a
homomorphism  $f:\hat\beta\to\beta$ such that $f_*([\hat\beta])=[z]$.
On passing to the extension $f^*\xi$, we may assume that $\beta$ is a $PD_2^+$-group.
 
Thirdly, suppose that $\beta$ is the $PD_2^+$-group of genus $g$
and $\eta_g$ is the central extension of $\beta$ by $\phi=\mathbb{Z}$
corresponding to a generator of $H^2(\beta;\mathbb{Z})$.
It is easy to see that $d^2_{2,0}$ and $\eta_{g*}$ are isomorphisms of infinite cyclic groups,
and so $d^2_{2,0}=\pm\eta_{g*}$.
There is a natural morphism of extensions from $\eta_g$ to $\xi$, 
and it follows that $d^2_{2,0}=\pm\xi_*$ whenever $\beta$ is a $PD_2^+$-group.

This is enough to show that $d^2_{2,0}([z])=\pm\xi_*([z])$ in general.
However to prove the theorem we must actually calculate $d^2_{2,0}$ 
for the special case $\eta_g$.
This  is a little tedious, but is not difficult.
We shall make one more reduction.
Let $h:\beta\to\mathbb{Z}^2$ be a degree-1 homomorphism.
Then $\eta_g=\pm{h^*\eta_1}$, 
and so it is enough to prove the claim for $\eta_1$.

Although $\pi=F(2)/F(2)_{[3]}$ has a presentation 
$\langle{x,y}\mid{x,y\leftrightharpoons[x,y]}\rangle$,
we shall use
\[
\langle{u,x,y}\mid{u[x,y],}~[u,x],~[u,y]\rangle
\]
instead.
We shall use the same notation $x$ and $y$ for generators of 
$\beta=\pi/\langle[x,y]\rangle=\mathbb{Z}^2$, for simplicity of reading.
Let $r$, $s$ and $t$ denote the relators $u[x,y]$, $[u,x]$ and $[u,y]$,
respectively.

Let $\Gamma=\mathbb{Z}[\pi]$ and $\Lambda=\mathbb{Z}[\beta]$.
The presentation of $\pi$ determines a (partial) resolution
\[P_*: \Gamma^3\to\Gamma^3\to\Gamma\to\mathbb{Z}\to0
\]
of the $\mathbb{Z}[\pi]$-augmentation module,
with bases $p_1^u$, $p_1^x$ and $p_1^y$ for $P_1$,
and $p_2^r$, $p_2^s$ and $p_2^t$ for $P_2$.
Let $\overline{P}_*=\Lambda\otimes_\Gamma{P_*}$, and
let $\partial'$ and $\overline{\partial}''$ be the differentials 
of $C_*^{FL}(\beta)$ and $\overline{P}_*$, respectively.

The homology LHS for the extension is based on the bicomplex 
$K_{p,q}=C_p^{FL}(\beta)\otimes_\Lambda\overline{P}_q$,
with differential $d=\partial'\otimes1+(-1)^p1\otimes\overline{\partial}''$.
(Since $\Lambda$ is commutative we may view the left module structure on
$C_*^{FL}(\beta)$ as also being a right module structure.)
The associated total complex $K^{tot}_n=\oplus _{p+q=n}K_{p,q}$
has an increasing filtration $F_iK^{tot}_*$,
given by 
\[
F_0K^{tot}_n=K_{0,n},\quad{F_1K^{tot}_n=K_{0,n}}\oplus{K_{1,n-1}}\quad
\mathrm{ and}\quad{F_2K^{tot}_n=K^{tot}_n}.
\]
(See Chapter VII of \cite{Br}. )

The generator of $H_2(\beta;\mathbb{Z})$ is represented 
by the 2-cycle $c^{[x,y]}_2$ in $C^{FL}_2(\beta)$.
Let $z=c^{[x,y]}_2\otimes1-c_1^x\otimes{p_1^y}+c_1^y\otimes{p_1^x}$.
Then $z\in{Z^2_{2,0}}$, and $z$ represents a generator of
$E^2_{2,0}=H_2(C_*^{FL}(\beta)\otimes_\beta\mathbb{Z})=H_2(\beta;\mathbb{Z})$.
Now 
\[
d^2_{2,0}([z])=[(1-x)p_1^y+(y-1)p_1^x]
\]
in
${E^2_{0,1}}=H_0(\beta;H_1(\overline{P}_*))\cong\phi.$
On the other hand, 
\[
\overline{\partial}''p_2^r=\partial_uu[x,y]p_1^u+\partial_xu[x,y]p_1^x
+\partial_yu[x,y]p_1^y
\]
\[
=p_1^u+(1-xyx^{-1})p_1^x+(x-1)p_1^y=p_1^u+(1-y)p_1^x+(x-1)p_1^y,
\]
since $u=1$ and $uxyx^{-1}=y$ in $\Lambda$. 
Thus 
\[
d^2_{2,0}([z])=[p_1^u-\overline{\partial}''p_2^r]=[p_1^u],
\]
which corresponds to the generator $u$ of $\phi$.
Hence $d^2_{2,0}=\eta_{1*}$, proving the theorem.
\end{proof}

It is easy to see that  $\mathrm{Ker}(d^2_{2,0})\leq\mathrm{Ker}(\xi_*)$
and $\mathrm{Im}(\xi_*)\leq\mathrm{Im}(d^2_{2,0})$, 
without such reductions or calculation.
The diagonal $\Delta:\pi\to{\pi\times_\beta\pi}$ splits the pullback 
$p^*\xi$ of $\xi$ over $p$.
Hence $p^*[\xi]=0$, 
and so 
\[
\mathrm{Ker}(d^2_{2,0})=\mathrm{Im}(H_2(p))\leq\mathrm{Ker}(\xi_*).
\]
On the other hand,
if $z=\Sigma_{r\in{R}}{n_r}c^r_2$ is a 2-cycle of
$\mathbb{Z}\otimes_\beta{C_*^{FL}(\beta)}$, then
$\Sigma{n_r}\varepsilon_x(r)=0$, for all $x\in{X}$.
Let $j$ be the inclusion of $\phi$ as a subgroup of $\pi$.
Since the exponent sum of $x$ in $\Pi{r^{n_r}}$ is 0,
for each $x\in{X}$, the image $js(\Pi{r^{n_r}})$ is in $\pi'$.
Since $\xi_*(z)=s(\Pi{r^{n_r}})$,
we see that 
\[
\mathrm{Im}(\xi_*)\leq\mathrm{Ker}(H_1(j))=\mathrm{Im}(d^2_{2,0}).
\]

\section{surface bundles with flat fibre}

When the base is an orientable surface and the fibre $F$ is the torus $T$ or 
the Klein bottle $Kb$ the class $[s(r)]$ is the only obstruction to a section.

\begin{theorem}
Let $p:E\to{B}$ be a bundle with base $B$ a surface and fibre $F=T$.
Let $\beta=\pi_1(B)$ and $\phi=\mathbb{Z}^2$,
and let $\theta$ be the action determined by conjugation in $\pi_1(E)$.
Then $p$ has a section if and only if  $[s(r)]=0$.
If $B$ is orientable  then $p$ has a section if and only if
$H_1(E;\mathbb{Z})\cong{H_0(B;H_1(F;\mathbb{Z}))}\oplus{H_1(B;\mathbb{Z})}$.
The $\phi$-conjugacy classes of sections are parametrized 
by $H^1(\beta;\phi^\theta)$.
\end{theorem}

\begin{proof}
Since the action of $Aut(\phi)=GL(2,\mathbb{Z})$ on $\phi$
is induced by the natural (based) action of $GL(2,\mathbb{Z})$ 
on $T=\mathbb{R}^2/\mathbb{Z}^2$,
every semidirect product $\mathbb{Z}^2\rtimes_\theta\beta$
is realized by a $T$-bundle over $B$ with a section.
Therefore $p$ has a section if and only if $p_*$ splits,
since bundles are determined by the associated extensions.
This in turn holds if and only if $[s(r)]=0$, since $Aut(\phi)=Out(\phi)$.

The second assertion follows from Lemma 2.

If $p_*$ splits and $\sigma$ and $\sigma'$ are two sections 
then $\sigma'(g)\sigma(g)^{-1}$ is in $\phi$, for all $g\in\beta$.
Therefore the sections are parametrized (up to conjugation by an element of $\phi$) by $H^1(\beta;\phi^\theta)$.
(See Proposition IV.2.3 of \cite{Br}.)
\end{proof}

If $p$ has a section then so does the pullback over $B^+$.
The converse also holds if $H^2(\beta;\phi^\theta)\cong{H_0(\beta;\mathbb{Z}^w\otimes\phi^\theta)}$ 
has no 2-torsion. 
For then restriction to $H^2(\beta^+;\phi^\theta)$ is injective,
since composition with the transfer is multiplication by 2.
(See \S9 of Chapter III of \cite{Br}.)

The situation is a little more complicated when $F=Kb$.
We may view $Kb$ as the quotient of $\mathbb{R}^2$ by a glide-reflection $x$
along the $X$-axis and a unit translation $y$ parallel to the $Y$-axis.
Then $\kappa=\pi_1(Kb)$ has presentation $\langle{x,y}\mid{xyx^{-1}=y^{-1}}\rangle$,
and $\zeta\kappa$ is generated by the image of $x^2$.
Let $\alpha$ and $\gamma$ be the automorphisms determined by
$\alpha(x)=x^{-1}$,  $\gamma(x)=xy$ and $\alpha(y)=\gamma(y)=y$.
Then  $Aut(\kappa)$ is generated by $\alpha$, $\gamma$ and $c_x$, and $\gamma^2=c_y$.
It is easily verified that $\alpha\gamma=\gamma\alpha$,
and so $Out(\kappa)\cong(Z/2Z)^2$ is the image of an abelian subgroup 
$\langle\alpha,\gamma\rangle<Aut(\kappa)$.

\begin{theorem}
Let $p:E\to{B}$ be a bundle with base $B$ a surface and fibre $F=Kb$.
Let $\beta=\pi_1(B)$ and $\phi=\kappa$,
and let $\theta$ be the action determined by conjugation in $\pi_1(E)$.
Then $p$ has a section if and only if $\theta$ factors through 
$Aut(\kappa)$ and $[s(r)]=0$. 
If $B$ is orientable then $p$ has a section if and only if $[s(r)]=0$.
\end{theorem}

\begin{proof}
The automorphism $\alpha$ is induced by a reflection across a circle of fixed points
(the image of the $Y$-axis), while the outer automorphism class of $\gamma$
is induced by a half-unit translation parallel to the $Y$-axis.
This may be isotoped  through homeomorphisms that commute 
with $\alpha$ to fix the $Y$-axis also.
Thus $\langle\alpha,\gamma\rangle$ lifts to a group 
of based self-homeomorphisms of $Kb$.
Hence $p$ has a section if and only if $p_*$ splits.
This in turn holds if and only if $\theta$ factors though 
$Aut(\kappa)$ and $[s(r)]=0$.

If $B$ is orientable then $\beta/\beta'$ is a free abelian group,
and so any homomorphism from $\beta$ to $Out(\kappa)$ factors through
$\langle\alpha,\gamma\rangle<Aut(\kappa)$.
Hence in this case $p$ has a section if and only if $[s(r)]=0$.
\end{proof}

If (as in Theorem 6) $\beta$ acts on $\zeta\phi$ through $w_1(\beta)$ 
we can make the condition $[s(r)]=0$  more explicit.
For then $H^2(\beta;\zeta\phi^\theta)$ maps injectively to 
$H^2(\beta^+;\zeta\phi)\cong\mathbb{Z}$ under 
passage to $\beta^+$.
Thus $p_*$ splits if and only if  $\theta$ factors through $Aut(\kappa)$
and the restriction to $p_*^{-1}(\beta^+)$ splits.
Since $\zeta\phi$ maps injectively to $\phi/I(\phi)\cong\mathbb{Z}$,
$H^2(\beta^+;\zeta\phi)$ in turn maps injectively to $H^2(\beta^+;\phi/I(\phi))$.
The image of $[\xi(p)]$ is the class of the extension 
\[1\to\phi/I(\phi)\to\hat\pi/I(\phi)\to\beta^+\to1,
\]
where $\hat\pi$ is the preimage of $\beta^+$.
Hence the extension is trivial if $\beta_1(\hat\pi)$ is odd, by Lemma 2.

\smallskip
{\it Examples}. Let $\pi$ be a discrete cocompact subgroup of $Nil^3\times\mathbb{R}$.
Then $\zeta\pi\cong\mathbb{Z}^2$ and $\pi/\zeta\pi\cong\mathbb{Z}^2$,
and so the coset space $E=\pi\backslash{Nil^3\times\mathbb{R}}$ 
is the total space of a $T$-bundle over $T$.
The action is trivial, and so the split extension is the product $\mathbb{Z}^4$.
Thus the bundle projection for this coset space has no section.
(In fact, $\pi/\pi'$ has rank 2, rather than 4, and so the criterion of \S5 fails.)
Similarly, coset spaces of discrete cocompact subgroups of $Nil^4$
are $T$-bundles over $T$ without sections.

The group with presentation
\[
\langle{u,v,x,y}\mid{u,v}\leftrightharpoons{x,y},
~[u,v]=x^2,~xyx^{-1}=y^{-1}\rangle
\]
is the group of a $\mathbb{N}il^3\times\mathbb{E}^1$-manifold  
which fibres over $T$ with fibre $Kb$.
The base group acts trivially on the fibre, but $\beta_1(\pi)=2$,
rather than 3, and so the bundle does not have a section.

The group with presentation
\[
\langle{u,v,x,y}\mid{u}\leftrightharpoons{x,y},~vxv^{-1}=x^{-1},~vy=yv,
~[u,v]=x^2,~xyx^{-1}=y^{-1}\rangle
\]
is the group of a flat 4-manifold  which fibres over $T$ with fibre $Kb$.
In this case $H^2(\beta;\zeta\phi)=Z/2Z$, but $[s(r)]\not=0$,
and so the bundle does not have a section.

Similar examples over a hyperbolic base $B$ may be constructed 
from these by pullback over a degree-1 map to $T$.

\section{bundles with hyperbolic fibre}

Suppose now that $\chi(F)<0$ or, equivalently, that $\zeta\phi=1$.
If an $F$-bundle $p:E\to{B}$ has a section then the action $\theta$ factors through $Aut(\phi)$.
Conversely, 
every semidirect product $\phi\rtimes_{\widetilde\theta}\beta$ is realized by a bundle with a section.
This follows from the work in \cite {Gr} extending \cite{EE} to the based cases.
We shall not use this, 
as our concern here is merely to give examples of such bundles without sections.

We may construct the extension corresponding to an action 
$\theta:\beta\to{Out(\phi)}$ as follows. 
Let $\langle{X}|r\rangle$ be a 1-relator presentation for $\beta$, 
with associated epimorphism $q:F(X)\to\beta$.
Let $\psi:F(X)\to{Aut}(\phi)$ be a lift of $\theta{q}$.
Then $\psi(r)=c_g$, for some $g\in\phi$, 
which is uniquely determined by $\psi$, since $\zeta\phi=1$.
Let $G=\phi\rtimes_\psi{F(X)}$.
Then $\pi=G/\langle\langle{rg^{-1}}\rangle\rangle$
is an extension of $\beta$ by $\phi$ which realizes the action $\theta$.
In particular, $\pi$ is a semidirect product if $g=1$.
However, $g$ depends on the choice of  $\psi$.
We need a condition which does not depend on this choice.

If such a bundle has a section then so does the associated Jacobian bundle,
with base $B$, fibre the Jacobian of $F$ and group $\pi/\phi'$.
Lemma 2 renders more explicit a result of Morita \cite{Mo}.
He showed that if $E$ and $F$ are orientable and $F$ is of genus $g\geq2$ 
then the Jacobian bundle has a section if and only if $\theta^*\mu=0$, 
where $\mu$ is a class in $H^2(\mathcal{M}_g;H^1(\phi;\mathbb{Z}))$.
(Here $\mathcal{M}_g$ is the mapping class group of $F$,
which is isomorphic to the orientation preserving subgroup of $Out(\phi)$, 
by a theorem of Nielsen.
See \cite{FM}.)
Examining his construction, 
we see that if $f$ is the 2-cocycle with values in $\phi^{ab}$
associated to a set-theoretic section $\sigma:\beta\to\pi/\phi'$,
as in \S6 above,
then $\theta^*\mu$ is the image of $[f]$ under the change of coefficient
isomorphism induced by the Poincar\'e duality isomorphism 
$\phi^{ab}\cong{H^1}(\phi;\mathbb{Z})$.
Thus if base and fibre are orientable the Jacobian bundle 
has a section if and only if $\pi^{ab}\cong(\phi/[\pi,\phi])\oplus\beta^{ab}$.
This is so if and only if $g\in[\pi,\phi]$, where $c_g=\psi(r)$,
for some (and hence all) $\psi$ as in the preceding paragraph.

Endo has suggested the following example of a surface bundle,
with base and fibre of genus 3, which has no section \cite{En}.
Let $D_1, D_2, D_3$ be disjoint small discs in the interior of the standard unit disc $D^2$,
and let $\Sigma=\overline{D^2\setminus\cup_{j\leq3}D_j}$ be the 4-punctured sphere,
with the standard planar orientation.
Let $F=T_3=\partial(\Sigma\times[0,1])\cong\Sigma_0\cup\Sigma_1$,
where $\Sigma_0$  and $\Sigma_1$ are collar neighbourhoods of $\Sigma\times\{0\}$
and $\Sigma\times\{1\}$, respectively,
meeting along $N=\partial\Sigma\times\{\frac12\}$.
Let $j_0$ and $j_1$ be the natural identifications of $\Sigma$ with $\Sigma_0$ and $\Sigma_1$,
respectively.
Orient $F$ so that $j_0$ is orientation preserving.
Then $j_1$ is orientation reversing.

Let $b_1,b_2,b_3,b_4$ be the boundary components of $\Sigma$,
and $x,y,z$ be simple closed curves in the interior of $\Sigma$, 
as in Figure 5.1 of \cite{FM}.
Let $d_1,\dots,d_4$ be simple closed curves parallel to $b_1,\dots,b_4$
in the interior of $\Sigma$.
The left hand Dehn twists $t_{d_1},\dots ,t_z$ about these curves fix $\partial\Sigma$.
The lantern relation asserts that 
\[
t_xt_yt_z=t_{d_1}t_{d_2}t_{d_3}t_{d_4},
\]
up to isotopy {\it rel\/} $\partial\Sigma$.
(See Chapter 5 of \cite{FM}.)

Let $x_i=j_i(x)$, and so on.
Then $t_{x_0}=j_0t_xj_0^{-1}$, etc, while $t_{x_1}=j_1t_x^{-1}j_1^{-1}$,
since the notions of left and right Dehn twist are interchanged under 
an orientation reversing homeomorphism.
Hence the lantern relation gives two equations
\[
t_{x_0}t_{y_0}t_{z_0}=t_{d_{10}}t_{d_{20}}t_{d_{30}}t_{d_{40}},
\]
and
\[
t_{x_1}^{-1}t_{y_1}^{-1}t_{z_1}^{-1}=
t_{d_{11}}^{-1}t_{d_{21}}^{-1}t_{d_{31}}^{-1}t_{d_{41}}^{-1},
\]
up to isotopy in $F$ {\it rel\/} $N$.
Combining the last two equations and using the commutativity of Dehn twists 
about disjoint curves gives
\[
t_{x_0}t_{y_0}t_{z_0}t_{x_1}^{-1}t_{y_1}^{-1}t_{z_1}^{-1}=
t_{d_{10}}(t_{d_{11}})^{-1}t_{d_{20}}(t_{d_{21}})^{-1}t_{d_{30}}(t_{d_{31}})^{-1}t_{d_{40}}(t_{d_{41}})^{-1}.
\]

Let $f$ be the hyperelliptic involution of $F$ which maps $\Sigma_0$ 
onto $\Sigma_1$.
This is orientation preserving, 
but induces an orientation-reversing involution of $N$, 
with two fixed points in each component of $N$.
Let $*$ be one of the two fixed points of $f$ on $b_1$, and let $g=[b_1]\in\phi=\pi_1(F,*)$.
Then $t_{d_{10}}(t_{d_{11}})^{-1}$ induces $c_g$  on $\phi$,
while $t_{d_{i1}}$ is isotopic to $t_{d_{i0}}$ {\it rel} $*$, for $i\geq2$.
If we modify Figure 5.1 of \cite{FM} so that $b_1$, 
$b_2$ and $b_3$ are aligned vertically down the $Y$-axis, 
and the fixed points of $f$ are the intersections of the boundary 
with this axis then we see that we may assume that
$f(x_0)=x_1$ and $f(y_0)=y_1$.
However, $f(z_0)=t_{y_1}^{-1}(z_1)$, i.e., $z_1=ft_{y_0}(z_0)$.
Thus $t_{x_1}=ft_{x_0}f^{-1}$, $t_{y_1}=ft_{y_0}f^{-1}$ and 
$t_{z_1}=ft_{y_0}t_{z_0}(ft_{y_0})^{-1}$.
The equation becomes
\[[t_{x_0},t_{y_0}t_{z_0}f][t_{y_0},t_{z_0}f][t_{z_0},ft_{y_0}]=c_g
\] 
in $Aut(\phi)$, the mapping class group of $(F,*)$.
The left hand side is a product of three commutators,
and so we may define an action $\theta:\beta\to{Out(\phi)}$
which sends the standard generators to 
$t_{x_0}$, $t_{y_0}t_{z_0}f$, $t_{y_0}$, $t_{z_0}f$, 
$t_{z_0}$ and $ft_{y_0}$, respectively.
(Thus $\mathrm{Im}(\theta)$ is generated by $t_{x_0},t_{y_0},t_{z_0}$ and $f$.)
It is not hard to see that the image of $g$ in $\phi/[\pi,\phi]\cong(Z/2Z)^4$ is nontrivial.
Hence $\pi^{ab}$ is a proper quotient of $\phi/[\pi,\phi]\oplus\beta^{ab}$,
and so the associated Jacobian bundle has no section.
Hence the $F$-bundle determined by $\theta$ does not have a section either.
However,
N.Salter has shown that there is a 2-fold cover of the base such that the induced $F$-bundle has a section (unpublished).

If we use the reflection $\rho$ of $F$ across $N$ instead of 
the hyperelliptic involution
then $t_{x_1}=\rho{t_{x_0}^{-1}}\rho^{-1}$, etc., 
and so we get the equation
\[
t_{x_0}t_{y_0}t_{z_0}\rho{t_{x_0}t_{y_0}t_{z_0}}\rho^{-1}=c_g.
\]
This gives rise to an $F$-bundle over $Kb$, 
with monodromy generated by $t_{x_0}t_{y_0}t_{z_0}$ and $\rho$,
and non-orientable total space.
We again see that the bundle has no section.
(The criterion of \S5 above involves no assumptions of orientability.)
However,
the bundle induced over the torus has cyclic monodromy, 
and so has a section.

The referees have suggested that Endo's idea may be extended 
in an inductive manner to give similar examples 
with base and fibre of genus $g$, for any $g\geq3$.
We should use the ``daisy chain relation" of \cite{EMV},
which is also the ``generalized lantern relation" of \cite{BKM},
and is an iteration of the lantern relation.

Are there any such examples with fibre of genus 2, 
or with hyperbolic fibre and base $T$?

\section{some questions}

The following questions mostly arise from the considerations of \S2-\S4 above.

\begin{enumerate}
\item{Is} the number of ways in which a $PD_4$-group $\pi$ with $\chi(\pi)>0$
can be an extension of surface groups bounded by $c^{\chi(\pi)}$ for some $c>1$?

\item{If} a torsion-free group is virtually the group of a surface bundle,
is it realized by an aspherical 4-manifold?

\item{Is} every iterated extension of $k\geq3$ 
surface groups realized by an aspherical $2k$-manifold?

\item{Is} every bundle with base and fibre hyperbolic surfaces
finitely covered by one which has a section?

\item{Which} bundle groups are realized by complex surfaces?

\item{If} a symplectic 4-manifold $M$ is homotopy equivalent to the total
space $E$ of a surface bundle is it diffeomorphic to $E$?

\end{enumerate}

Salter has shown that for each $n\geq1$ there is a $PD_4^+$-group 
$\pi$ with $\chi(\pi)=24n-8$ which has at least $2^n$ normal subgroups $K$ 
such that $K$ and $\pi/K$ are $PD_2^+$-groups \cite{Sa}.
On the other hand,
Corollary 5.6.1 of \cite{Hi} can be used to show that 
if $\chi(\pi)=4d$ then there are at most $d^{2d+3}$ such subgroups.
(Moreover, if $\chi(\pi)\geq16$ then at most $2^{\chi(\pi)}$ 
of these have $|\chi(\pi/K)|>\log\chi(\pi)$.)

Question 1 is only of interest for surface bundle groups of type III.
The same is true for Question 2, since the answer is yes if the
group is of type I or II, 
by Corollary 7.3.1 and Theorem 9.9 of \cite{Hi}, respectively.

Johnson has shown that every iterated extension of surface groups 
has a subgroup of finite index which is the fundamental group of an
aspherical compact smooth manifold \cite{Jo98}.
A natural extension of Question 3, which would include Question 2, 
is whether every torsion-free group which is virtually 
an iterated extension of surface groups is thus realizable.

Question 4 is closely related to Problem 2.17 of Kirby's list,
which asks whether every such bundle with $\theta$ injective
has a multisection \cite{Ki}.
(A {\it multi-section\/} of  $p$ is a surface $C\subset{E}$ such that 
$p|_C:C\to{B}$ is a finite covering projection.
In terms of groups, $p$ has a multisection if $\beta$ has a subgroup $\gamma$
of finite index such that $\theta|_\gamma$ factors through $Aut(\phi)$.
Clearly every group of type II has a multisection.)

There are only finitely many surface bundle groups $\pi$ 
with given $\chi(\pi)>0$ which are realizable by
holomorphic submersions $p:S\to{C}$, 
where $S$ is a complex surface and $C$ a complex curve,
by the geometric Shafarevitch conjecture, 
proven by Parshin and Arakelov.
Moreover, ``Kodaira fibrations" (for which $\pi$ is not virtually a product)
have at most finitely many holomorphic sections,
by work of Manin and Grauert.
(See \cite{McM} for an illuminating sketch of these results, 
and for further references.)
However, in Question 6 we allow also the possibility that 
there are complex surfaces with fundamental group an extension of $PD_2$-groups,
but for which the minimal models are not aspherical.

The final question is prompted by the facts that it is true 
when $M$ is a complex surface, and that  total spaces of 
surface bundles are sympletic.

\end{document}